\documentclass{amsart}
\usepackage[latin1]{inputenc}
\usepackage{amsmath}
\usepackage{amsthm}
\usepackage{amsfonts}
\usepackage{amssymb}
\usepackage[T1]{fontenc}
\usepackage{float}
\usepackage{wasysym}
\usepackage[all]{xy}
\usepackage{algorithm,algorithmic}

\newtheorem{theorem}{Theorem}[section]
\newtheorem{prop}[theorem]{Proposition}
\newtheorem{lem}[theorem]{Lemma}
\newtheorem{corol}[theorem]{Corollary}

\theoremstyle{definition}
\newtheorem{defi}[theorem]{Definition}

\newtheorem{rmq}[theorem]{Remark}

\def\Gr{{\rm{Gr}}}

\def\dim{{\rm{dim}\,}}
\def\ddim{{\textbf{dim}\,}}

\def\modd{\rm{mod-}}

\def\<{\left<}
\def\>{\right>}

\def\ens#1{\left\{ #1 \right\}}
\def\fl{{\longrightarrow}\,}

\def\CC{{\mathcal{C}}}
\def\A{{\mathbb{A}}}

\def\C{{\mathbb{C}}}
\def\D{{\mathbb{D}}}
\def\E{{\mathbb{E}}}

\def\P{{\mathbb{P}}}
\def\Q{{\mathbb{Q}}}
\def\Z{{\mathbb{Z}}}

\def\Aaffine{\widetilde{\A}}
\def\Daffine{\widetilde{\D}}
\def\Eaffine{\widetilde{\E}}

\def\ens#1{\left\{ #1 \right\}}
\def\Ext{{\rm{Ext}}}
\def\End{{\rm{End}}}
\def\Hom{{\rm{Hom}}}
\def\Ob{{\rm{Ob}}}
\def\d{{\partial}}

\def\ker{{\rm{Ker}}\,}
\def\coker{{\rm{Coker}}\,}

\def\Aaffine{\tilde{\mathbb A}}
\def\Ai{\mathbb A_{\infty}}

\def\ev_1{{\rm{ev}}}

\def\k{{\bf{k}}}
\def\kQ{{\bf{k}}Q}

\title[Euclidean friezes]{Friezes and a construction of the euclidean cluster variables}
\author{I. Assem and G. Dupont}
\date{\today}

\begin{document}

\begin{abstract}
	Let $Q$ be an euclidean quiver. Using friezes in the sense of Assem-Reutenauer-Smith, we provide an algorithm for computing the (canonical) cluster character associated to any object in the cluster category of $Q$. In particular, this algorithm allows to compute all the cluster variables in the cluster algebra associated to $Q$. It also allows to compute the sum of the Euler characteristics of the quiver grassmannians of any module $M$ over the path algebra of $Q$.
\end{abstract}

\maketitle

\setcounter{tocdepth}{1}
\tableofcontents

\section{Introduction}
	\subsection{Cluster algebras}
		In \cite{cluster1}, Fomin and Zelevinsky introduced a class of commutative algebras, called \emph{cluster algebras}, in order to design a combinatorial framework for studying total positivity in algebraic groups and canonical bases in quantum groups. Since then, cluster algebras have shown connections with various areas of mathematics like Lie theory, Teichm\"uller theory, Poisson geometry or the representation theory of algebras.

		Let $Q=(Q_0,Q_1)$ be an acyclic quiver, that is an oriented graph which contains no oriented cycle. Let $Q_0$ denote the set of points in $Q$ and $Q_1$ the set of arrows in $Q$. We assume that these sets are finite and we fix a $Q_0$-tuple $\textbf u=(u_i|i \in Q_0)$ of indeterminates. The cluster algebra $\mathcal A(Q)$ with initial seed $(Q,\textbf u)$ is a subalgebra of $\Z[\textbf u^{\pm 1}]=\Z[u_i^{\pm 1}|i \in Q_0]$, generated by a distinguished set of generators, called \emph{cluster variables} defined recursively using a combinatorial process called \emph{mutation} \cite{cluster1}. 

		The number of cluster variables in a cluster algebra can be finite or infinite. If the cluster algebra has finitely many cluster variables, it is of \emph{finite type}. This is the case if and only if $Q$ is a Dynkin quiver, that is a quiver of type $\A,\D$ or $\E$ \cite{cluster2}. In this case, it is possible to compute algorithmically the cluster variables in $\mathcal A(Q)$ by applying mutations recursively. If $Q$ is not a Dynkin quiver, this recursive process does not end. Nevertheless, cluster variables can be parametrised using the representation theory of $Q$. In this paper, we exhibit an algorithm for constructing cluster variables with respect to this parametrisation when $Q$ is a quiver of euclidean type, that is an acyclic orientation of an \emph{euclidean} diagram of type $\Aaffine,\Daffine$ or $\Eaffine$ (these diagrams are also called \emph{extended Dynkin} or \emph{affine} in the literature).

		More precisely, if $\CC_Q$ denotes the \emph{cluster category} introduced in \cite{BMRRT}, Caldero, Chapoton and Keller introduced an explicit map 
		$$X_?: \Ob(\CC_Q) \fl \Z[\textbf u^{\pm 1}]$$
		called (canonical) \emph{cluster character} (also called \emph{Caldero-Chapoton map} in the literature) such that 
		$$\ens{\textrm{cluster variables in }\mathcal A(Q)} = \ens{ X_M | M \textrm{ indecomposable and rigid in } \CC_Q}.$$
		where an object is called \emph{rigid} if it has no self-extension.

		In general, in order to compute the character $X_M$ of an object $M$, one has to compute the Euler characteristics of certain projective varieties called \emph{quiver grassmannians} (see Section \ref{section:Xfrieze} for definitions). These Euler characteristics are hard to compute in practice. In some very simple cases, explicit computations were made \cite{CZ,Cerulli:A21,Cerulli:string,Poettering:string} but in general there exists no simple way for computing these characteristics and thus for computing $X_M$.

	\subsection{Friezes}
		In this article, we provide a simple combinatorial method for computing these characters when $Q$ is an euclidean quiver. Our methods are based on the concept of \emph{frieze} on a stable translation quiver. This concept appeared in various independent works under different names~: \emph{frises} in \cite{ARS:frises}, \emph{frieze patterns} in \cite{Propp:frieze,CoxeterConway1,CoxeterConway2}, \emph{cluster-mesh algebras} in \cite{Dupont:stabletubes}. 

		\subsubsection{Friezes on translation quivers}
			By definition, a \emph{stable translation quiver} $(\Gamma, \tau)$ is a quiver $\Gamma=(\Gamma_0,\Gamma_1)$ without loops together with a bijection $\tau:\Gamma_0 \fl \Gamma_0$, called \emph{translation} such that there is a bijection from the set of arrows $j \fl i$ in $\Gamma_1$ to the set of arrows $\tau i \fl j$ in $\Gamma_1$ for any two points $i,j \in \Gamma_0$. We always assume that a translation quiver is locally finite, that is, for any $i \in Q_0$, there exist only finitely many arrows starting or ending at $i$ in $\Gamma_1$. Note that $\Gamma_0$ and $\Gamma_1$ may be infinite sets.

			Throughout the paper, commutative rings are supposed unitary and, given a commutative ring $R$, we denote by $1$ its identity.\begin{defi}\label{defi:frieze}
				Let $(\Gamma,\tau)$ be a stable translation quiver and $R$ be a commutative ring. An \emph{$R$-frieze} on $(\Gamma,\tau)$ is a map $f: \Gamma_0 \fl R$ such that, for any $i,j \in \Gamma_0$, 
				$$f(i)f(\tau(i))=\prod_{\substack{\alpha \in \Gamma_1\\\alpha : j \fl i}}f(j)+1.$$

				By a \emph{frieze} on $(\Gamma,\tau)$, we mean an $R$-frieze for a certain commutative ring $R$.
			\end{defi} 
			It is clear from the definition that an $R$-frieze on a stable translation quiver $(\Gamma,\tau)$ is entirely determined by its restriction on each of the connected components of $(\Gamma,\tau)$.

		\subsubsection{Friezes on Auslander-Reiten quivers}
			If $\CC$ is a Krull-Schmidt category with an Auslander-Reiten translation, we denote by $\Gamma(\CC)$ its Auslander-Reiten quiver. It is naturally a translation quiver for the translation induced by the Auslander-Reiten translation in $\CC$.

			If $f$ is a frieze on $\Gamma(\CC)$, we slightly abuse notations and view $f$ as a function on objects in $\CC$. For any object $M$ in $\CC$, we define $f(M)$ as follows. If $M$ is indecomposable, $f(M)$ is the value of $f$ on the point in $\Gamma(\CC)$ corresponding to the isomorphism class of $M$. If $M_1, \ldots, M_n$ are indecomposable, we define $f(M_1 \oplus \cdots \oplus M_n)$ as the product $f(M_1)\cdots f(M_n)$ and by convention, we set $f(0)=1$.
	
	\subsection{Auslander-Reiten quivers of cluster categories}
		We fix an algebraically closed field $\k$ and an acyclic quiver $Q=(Q_0,Q_1)$ such that $Q_0$ and $Q_1$ are finite. We denote by $\kQ$ the path algebra of $Q$ and by mod-$\kQ$ the category of finitely generated right-$\kQ$-modules. For any $i \in Q_0$, we denote by $S_i$ the simple module associated to $i$, by $P_i$ its projective cover and by $I_i$ its injective hull. We denote by $\tau$ the Auslander-Reiten translation on mod-$\kQ$. Note that a $\kQ$-module $M$ will be identified with a representation $(M(i),M(\alpha))_{i \in Q_0,\alpha \in Q_1}$ of $Q$.

		Let $D^b(\modd \kQ)$ be the bounded derived category. We also denote by $\tau$ the Auslander-Reiten translation on $D^b(\modd \kQ)$ and we denote by $[1]$ the suspension functor in $D^b(\modd \kQ)$. Following \cite{BMRRT}, the \emph{cluster category} of $Q$ is the orbit category $\CC_Q$ of the functor $\tau^{-1}[1]$ in $D^b(\modd \kQ)$. It is a triangulated category satisfying the 2-Calabi-Yau property, that is, there is a bifunctorial duality
		$$\Ext^1_{\CC_Q}(X,Y) \simeq D \Ext^1_{\CC_Q}(Y,X)$$
		for any $X,Y \in \Ob(\CC_Q)$ where $D=\Hom_{\k}(-,\k)$ \cite{K,BMRRT}.

		The set of indecomposable objects in $\CC_Q$ can be identified with the disjoint union of the set of indecomposable $\kQ$-modules and the set of shifts of the indecomposable projective $\kQ$-modules \cite{BMRRT}. 

		Under this identification, the Auslander-Reiten quiver $\Gamma(\CC_Q)$ of the cluster category $\CC_Q$ is a stable translation quiver for the translation induced by the shift $[1]$ on $\CC_Q$ which coincides with the Auslander-Reiten translation $\tau$ of induced by the Auslander-Reiten translation on $D^b(\modd \kQ)$. 

		All the projective $\kQ$-modules belong to the same connected component $\mathcal P$ of $\Gamma(\CC_Q)$. This component, called the \emph{transjective component}, is isomorphic to $\Z Q$. If $Q$ is a Dynkin quiver, $\mathcal P$ is the unique connected component in $\Gamma(\CC_Q)$. Otherwise, $\Gamma(\CC_Q)$ contains infinitely many other connected components, called \emph{regular}. An indecomposable object in the transjective component is called \emph{transjective}, it is of the form $P_i[n+1]$ for some $n \in \Z$ and $i \in Q_0$. If $n < 0$, then $P_i[n+1]$ is identified with the $\kQ$-module $\tau^{-n}P_i$ and is called \emph{postprojective}, if $n > 0$, then $P_i[n+1]$ is identified with the $\kQ$-module $\tau^{n-1}I_i$ and is called \emph{preinjective}. An indecomposable object in a regular component is identified with a regular $\kQ$-module and is called \emph{regular}. A decomposable object is \emph{transjective} (or \emph{regular}) if all its direct summands are transjective (or regular, respectively).

		If $Q$ is euclidean, regular components are parametrised by $\P^1 (\k)$ and for any $\lambda \in \P^1 (\k)$ the corresponding component $\mathcal T_\lambda$ is a \emph{tube} of rank $p_\lambda \geq 1$. If $p_\lambda=1$, $\mathcal T_\lambda$ is called \emph{homogeneous}, otherwise it is called \emph{exceptional}. We set 
		$$\P^{\mathcal E}=\ens{\lambda \in \P^1 (\k) | p_\lambda > 1}$$
		the index set of exceptional tubes and 
		$$\P^{\mathcal H}=\ens{\lambda \in \P^1 (\k) | p_\lambda = 1}$$
		the index set of homogeneous tubes. There exists at most three exceptional tubes in $\Gamma(\CC_Q)$, in other words $|\P^{\mathcal E}| \leq 3$. A module at the mouth of a tube is called \emph{quasi-simple}. 

		If $Q$ is euclidean, the Tits form of $Q$ on $\Z^{Q_0}$ has corank one and there exists a unique positive generator of its radical which is denoted by $\delta$. The vector $\delta$ is called the \emph{radical vector} or \emph{positive minimal imaginary root} of $Q$ and it is known that, for any orientation of $Q$, there exists a point $e \in Q_0$ which is either a sink or a source and such that $\delta_e=1$. Throughout the paper $e$ denotes such a point.

	\subsection{Main results}
		Let $Q$ be an euclidean quiver. For any $\lambda \in \P^1(\k)$, there exists a unique quasi-simple module $M_\lambda$ in $\mathcal T_\lambda$ such that $\dim M_\lambda(e) =1$ and we set  
		$$N_\lambda = \left\{\begin{array}{ll}
			M_\lambda[1] & \textrm{ if $e$ is a source}~;\\
			M_\lambda[-1] & \textrm{ if $e$ is a sink}.
		\end{array}\right.$$
		We prove in Lemmas \ref{lem:extension1inj} and \ref{lem:extension1proj} the existence of transjective objects $B_\lambda$ and $B_\lambda'$ uniquely determined by the fact that there exist non-split triangles
		$$N_\lambda \fl B_\lambda \fl S_e \fl N_\lambda[1] \textrm{ and } S_e \fl B_\lambda' \fl N_\lambda \fl S_e[1]$$
		if $e$ is a source and non-split triangles
		$$N_\lambda \fl B_\lambda' \fl S_e \fl N_\lambda[1] \textrm{ and } S_e \fl B_\lambda \fl N_\lambda \fl S_e[1]$$
		if $e$ is a sink.

		\begin{defi}
			Let $Q$ be an euclidean quiver. A \emph{cluster frieze} on $\Gamma(\CC_{Q})$ is a frieze $f$ on $\Gamma(\CC_{Q})$ with values in a commutative ring $R$ such that for any $\lambda \in \P^{\mathcal E}$
			\begin{equation}\label{eq:clusterfrieze}
				f(N_{\lambda}[k])=\frac{f(B_{\lambda}[k])+f(B'_{\lambda}[k])}{f(S_e[k])} \textrm{ for any } k \in \ens{0,\ldots, p_\lambda-1}.
			\end{equation}

			A cluster frieze is called a \emph{strong cluster frieze} if (\ref{eq:clusterfrieze}) holds for any $\lambda \in \P^1(\k)$. 
		\end{defi}

		Cluster friezes allow to realise cluster variables in the following sense~:
		\begin{theorem}\label{theorem:main}
			Let $\CC_{Q}$ be the cluster category of an euclidean quiver $Q$ and let $f$ be a cluster frieze on $\Gamma(\CC_{Q})$ with values in $\Q(u_i,i \in Q_0)$ such that $f(P_i[1])=u_i$ for any $i \in Q_0$. Then $f(M)=X_M$ for any indecomposable rigid object $M$ in $\CC_Q$.

			In particular, $f$ induces a bijection between the set of indecomposable rigid objects in $\CC_Q$ and the set of cluster variables in $\mathcal A(Q,\textbf u)$.
		\end{theorem}

		We actually prove a slightly more general result. Namely, we prove that if $f$ is a cluster frieze on $\Gamma(\CC_Q)$ such that $f(P_i[1])=u_i$ for any $i \in Q_0$, then $f$ and $X_?$ coincide on the transjective component and the exceptional tubes of $\Gamma(\CC_Q)$.

		Strong cluster friezes allow moreover to realise cluster characters associated to modules in homogeneous tubes~:
		\begin{theorem}\label{theorem:strong}
			Let $\CC_{Q}$ be the cluster category of an euclidean quiver $Q$ and let $f$ be a strong cluster frieze on $\Gamma(\CC_{Q})$ with values in $\Q(u_i,i \in Q_0)$ such that $f(P_i[1])=u_i$ for any $i \in Q_0$. Then $f(M)=X_M$ for any indecomposable object $M$ in $\CC_Q$.
		\end{theorem}

		Note that, if the objects $B_\lambda,B_\lambda'$ can be explicitly constructed, Theorem \ref{theorem:main} (and Theorem \ref{theorem:strong}) provide a purely combinatorial method for computing every cluster variable (and every cluster character, respectively) in a cluster algebra of euclidean type. In Section \ref{section:explicit}, we provide such explicit constructions for any euclidean quiver.

	\subsection{Organisation of the paper}
		In Section \ref{section:Xfrieze}, we prove that the cluster character $X_?$ induces a frieze on the Auslander-Reiten quiver of any cluster category. Section \ref{section:proof} is dedicated to the proof of Theorems \ref{theorem:main} and \ref{theorem:strong}. As an application, we use friezes of numbers in Section \ref{section:GrM} in order to compute the Euler characteristics of the so-called \emph{complete quiver grassmannians}. Section \ref{section:explicit} provides a case-by-case analysis in order to construct explicitly the objects $B_\lambda$ and $B_\lambda'$. As an application, we provide an algorithm for computing all cluster variables in cluster algebras associated to euclidean quivers in Section \ref{section:algorithms}. In Section \ref{section:varE6} we use this algorithm to compute explicitly the regular cluster variables in type $\Eaffine_6$. Finally, in Section \ref{section:GrMtypeE}, we use it to compute the Euler characteristics of the complete quiver grassmannians of quasi-simple regular modules in types $\Eaffine_6$, $\Eaffine_7$ and $\Eaffine_8$.

\section{Cluster characters and friezes}\label{section:Xfrieze}
	We recall the definition of the cluster character associated to an acyclic quiver $Q$. For any $\kQ$-module $M$, we denote by $\ddim M \in \Z_{\geq 0}^{Q_0}$ its dimension vector. For any $\textbf e \in \Z_{\geq 0}^{Q_0}$, the \emph{quiver grassmannian of $M$ of dimension $\textbf e$} is 
	$$\Gr_{\textbf e}(M)=\ens{N \subset M | \ddim N = \textbf e}.$$
	This is a projective variety and we denote by $\chi(\Gr_{\textbf e}(M))$ its Euler characteristic with respect to the simplicial (or \'etale) cohomology if $\k = \C$ (or if $\k$ is arbitrary, respectively).
	
	We denote by $\<-,-\>$ the Euler form on $\kQ$-mod, which here coincides with the Tits form, given by
	$$\<M,N\>=\dim \Hom_{\kQ}(M,N) - \dim \Ext^1_{\kQ}(M,N).$$
	It only depends on $\ddim M$ and $\ddim N$ (see \cite{ASS}).

	\begin{defi}[\cite{CC}]
		Let $Q$ be an acyclic quiver. The \emph{cluster character associated to $Q$} is the map 
		$$X_?:\Ob(\CC_Q) \fl \Z[\textbf u^{\pm 1}]$$
		defined as follows~: 
		\begin{enumerate}
			\item[$(\iota)$] If $M$ is an indecomposable $\kQ$-module
				\begin{equation}\label{eq:XM}
					X_M=\sum_{\textbf e \in \Z_{\geq 0}^{Q_0}} \chi(\Gr_{\textbf e}(M)) \prod_{i \in Q_0} u_i^{-\<\textbf e,S_i\>-\<S_i, \ddim M- \textbf e\>}~;
				\end{equation}
			\item[$(\iota\iota)$] If $M \simeq P_i[1]$, then
				$$X_M=X_{P_i[1]}~;$$
			\item[$(\iota\iota\iota)$] For any two objects $M,N \in \Ob(\CC_Q)$,
				$$X_{M \oplus N}=X_MX_N.$$
		\end{enumerate}
	\end{defi}
	Note that (\ref{eq:XM}) also holds for decomposable $\kQ$-modules \cite{CC}. Moreover $X_M=1$ if $M = 0$.

	It was proved in \cite{CK2} that $X_?$ induces a 1-1 correspondence between the set of indecomposable rigid objects in $\CC_Q$ and the set of cluster variables in $\mathcal A(Q)$. Thus, a cluster variable $x \in \mathcal A(Q)$ is called \emph{transjective} (or \emph{regular}) if $x=X_M$ for some indecomposable rigid transjective (or indecomposable rigid regular, respectively) object $M$ in $\CC_Q$.

	\begin{prop}\label{prop:Xfrieze}
		Let $Q$ be an acyclic quiver. Then the cluster character associated to $Q$ induces a $\Z[\textbf u^{\pm 1}]$-frieze on $\Gamma(\CC_{Q})$.
	\end{prop}
	\begin{proof}
		Let $M$ be an indecomposable object in $\CC_Q$. Then either $M$ is regular or $M$ is transjective. If $M$ is a regular $\kQ$-module, it follows from \cite[Proposition 3.10]{CC} that $X_MX_{\tau M}=X_B+1$ where 
		$$0 \fl \tau M \fl B \fl M \fl 0$$
		is an almost split sequence of $\kQ$-modules. It thus follows that $X_?$ induces a frieze on the regular component $\mathcal T$ containing $M$. Explicitly, the frieze is given by $f(i,n)=X_{R_i^{(n)}}$ for any $i \in \Z/p\Z$ and $n \geq 1$ where the $R_i$ for $i \in \Z/p\Z$ denote the quasi-simple modules in $\mathcal T$ ordered in such a way that $\tau R_i \simeq R_{i-1}$ for any $i \in \Z/p\Z$ and $R_i^{(n)}$ is the unique indecomposable regular $\kQ$-module with quasi-socle $R_i$ and quasi-length $n$.

		If $M$ is transjective, 
		$$\Ext^1_{\CC_Q}(M[1],M) \simeq \Hom_{\CC_Q}(M[1],M[1]) \simeq \End_{\CC_Q}(M[1]) \simeq \k.$$
		It follows from \cite[Theorem 2]{CK2} that 
		$$X_MX_{M[1]}=X_B + X_{B'}$$
		where $B$ and $B'$ are the unique objects such that there exist non-split triangles
		$$M[1] \fl B \fl M \fl M[2] \textrm{ and }M \fl B' \fl M[1] \fl M[1].$$
		Since $\End_{\CC_Q}(M[1]) \simeq \k$, $B'=0$ and thus $X_MX_{M[1]}=X_B+1$. Moreover, since $\Ext^1_{\CC_Q}(M[1],M) \simeq \k,$ the non-split triangle $M[1] \fl B \fl M \fl M[2]$ is almost split. Thus, $X_?$ induces a frieze on the transjective component. 

		It follows that $X_?$ induces a frieze on each connected component of $\Gamma(\CC_Q)$ and thus on $\Gamma(\CC_Q)$. By definition, $X_M \in \Z[\textbf u^{\pm 1}]$ for any object $M$ in $\CC_Q$ so that the frieze induced by $X_?$ takes its values in the commutative unitary ring $\Z[\textbf u^{\pm 1}]$. This completes the proof.
	\end{proof}

\section{Proofs of Theorems \ref{theorem:main} and \ref{theorem:strong}}\label{section:proof}
	The proofs consist of several steps. First, we prove that if $Q$ is any acyclic quiver and $M$ is a transjective object $M$, then $X_M$ can be computed via a frieze on the transjective component. If $Q$ is euclidean, we prove in Section \ref{ssection:tubes} that a frieze on a tube in $\Gamma(\CC_Q)$ is entirely determined by its values at the mouth of the tube. In Section \ref{ssection:qs}, we give relations between cluster characters associated to quasi-simple modules in tubes and cluster characters associated to transjective objects. All these results are applied in Sections \ref{ssection:mainproof} and \ref{ssection:strongroof} in order to prove Theorems \ref{theorem:main} and \ref{theorem:strong}.

	\subsection{The transjective component}\label{ssection:transjective}
		Given a locally finite quiver $Q$, we denote by $\Z Q$ its \emph{repetition quiver}, that is, the quiver whose points are labelled by $\Z \times Q_0$ and arrows are given by $(n,i) \fl (n,j)$ and $(n-1,j) \fl (n,i)$ if $i \fl j \in Q_1$. It is a stable translation quiver for the translation $\tau (n,i) = (n-1,i)$ for any $n \in \Z$, $i \in Q_0$. A \emph{section} in a repetition quiver $\Z Q$ is a full connected subquiver $\Sigma$ in $\Z Q$ such that $\Z \Sigma \simeq \Z Q$ (see \cite{ASS}).

		\begin{lem}\label{lem:asection}
			Let $Q$ be an acyclic quiver and $R$ be a commutative ring. Let $f$ be a $R$-frieze on $\Z Q$. Then $f$ is entirely determined by the image of a section in $\Z Q$.
		\end{lem}
		\begin{proof}
			Let $f$ and $g$ be $R$-friezes on $\Z Q$ such that $f$ and $g$ coincide on a section $\Sigma$ in $\Z Q$. Since $\Z Q \simeq \Z \Sigma$, we can label the points in $\Z Q$ by $\Z \times \Sigma_0$ such that points in $\Sigma$ correspond to $\ens 0 \times \Sigma_0$. Since $\Sigma$ is acyclic, we define a partial order on $\Sigma_0$ by taking the closure $\prec$ of the order $i \prec_e j$ if $j \fl i$ in $Q_0$. 
		
			We now prove that for any $n \geq 0$ and any $i \in \Sigma_0$, $f(n,i)=g(n,i)$. We put on $\Z_{\geq 0} \times \Sigma_0$ the lexicographic order $\prec$ induced by $\leq$ and $\prec$. We prove the claim by induction on this lexicographic order. If $n=0$, the result holds by hypothesis. Fix thus $n \geq 1$ and $i \in \Sigma_0$ and assume that $f(m,j)=g(m,j)$ for any $(j,m) \prec (i,n)$. By definition of the frieze we have 
			$$f(n,i)f(n-1,i)=\prod_{(m,j) \fl (n,i)} f(m,j)+1$$
			and 
			$$g(n,i)g(n-1,i)=\prod_{(m,j) \fl (n,i)} g(m,j)+1.$$
			But if $(m,j) \fl (n,i)$, then it follows from the definition of $\Z\Sigma$ that either $m=n-1$ and $j \fl i$ or $m=n$ and $i \fl j$. Hence, in any case, $(m,j) \prec (n,i)$. Thus,
			$$f(n,i)=\frac{\prod_{(m,j) \fl (n,i) \in \Sigma_1} f(m,j)+1}{f(n-1,i)}=\frac{\prod_{(m,j) \fl (n,i) \in \Sigma_1} g(m,j)+1}{g(n-1,i)}=g(n,i).$$

			If $n \leq 0$, we use the same argument with the dual order on $\Z_{\leq 0} \times \Sigma_0$.
		\end{proof}

		\begin{corol}\label{corol:transjective}
			Let $Q$ be an acyclic quiver. Let $f$ be a frieze on $\Gamma(\CC_Q)$ such that $f(P_i[1])=u_i$ for every $i \in Q_0$. Then for any indecomposable transjective object $M$, 
			$$f(M)=X_M.$$
			In particular, $f$ induces a bijection between the set of indecomposable transjective objects in $\CC_Q$ and the set of transjective cluster variables in $\mathcal A(Q,\textbf u)$.
		\end{corol}
		\begin{proof}
			The transjective component $\mathcal P$ of $\Gamma(\CC)$ is isomorphic to $\Z Q$. The full subquiver $\Sigma$ whose points correspond to $P_i[1]$, with $i \in Q_0$, is a section in $\Z Q$. By Proposition \ref{prop:Xfrieze}, the restrictions of $X_?$ and $f$ to $\mathcal P$ are friezes and it follows from Lemma \ref{lem:asection} that they are entirely determined by their values on $\Sigma$. Since $f$ and $X_?$ coincide on $\Sigma$, they coincide on $\mathcal P$. 

			The second statement follows from \cite[Theorem 4]{CK2}.
		\end{proof}

		\begin{rmq}
			For acyclic quivers, Corollary \ref{corol:transjective} provides a representation-theoretic interpretation of the combinatorial construction of cluster variables with friezes provided in \cite{ARS:frises}. A similar interpretation will be provided for acyclic valued graphs in \cite{DN:nsl}.
		\end{rmq}

	\subsection{Friezes on tubes}\label{ssection:tubes}
			Let $\Ai$ be the locally finite quiver with point set $\Z_{>0}$ and arrows $i \fl i+1$ for any $i \in \Z_{>0}$. Let $\Z \Ai$ be its repetition quiver. We recall that a (stable) \emph{tube} $\mathcal T$ is a translation quiver isomorphic to $\Z\Ai/(\tau^p)$ for some integer $p \geq 1$, called the \emph{rank} of $\mathcal T$. Points in $\mathcal T$ can be labelled by $\Z/p\Z \times \Z_{>0}$ and the translation on $\mathcal T$ induced by the translation on $\Z\Ai$ is thus given by $\tau (i,n) = (i-1,n)$ for any $i \in \Z/p\Z$, $n \in \Z_{>0}$. The \emph{mouth} of a tube $\mathcal T$ is the set of points $\ens{(i,1)| i \in \Z/p\Z}$.

			\begin{defi}[\cite{Dupont:stabletubes}]
				Let $T_i, i \in \Z_{\geq 0}$ be a family of indeterminates over $\Z_{\geq 0}$. We define the \emph{$n$-th generalised Chebyshev polynomial} $P_n$ as follows. First, $P_{-1}=0$, $P_{0}=1$ and for any $n \geq 1$, $P_n$ is the polynomial in $\Z[T_0, \ldots, T_{n-1}]$ given by 
				$$P_{n+1}(T_0, \ldots, T_n)=T_n P_n(T_0, \ldots, T_{n-1})-P_{n-1}(T_0, \ldots, T_{n-2}).$$
			\end{defi}

			Note that, for any $n \geq 1$, 
			$$P_n(T_0,\ldots, T_{n-1})
			=\det\left[\begin{array}{ccccc}
				T_{n-1} & 1 & & & (0)\\
				1 & \ddots & \ddots & \\
				& \ddots & \ddots & \ddots  \\
				&& \ddots & \ddots & 1 \\
				(0) &&& 1 & T_0
			\end{array}\right] \in \Z[T_0, \ldots, T_{n-1}].$$

			If $R$ is an arbitrary commutative unitary ring. The ring homomorphism $\Z \fl R$ sending $1$ to $1_R$ allows to view each generalised Chebyshev polynomial $P_n$ as a polynomial function in $n$ variables with coefficients in $R$. It was observed in \cite{Dupont:stabletubes} that friezes on tubes are entirely governed by generalised Chebyshev polynomials in the following sense.
			\begin{theorem}[\cite{Dupont:stabletubes}]\label{theorem:Chebyshev}
				Let $\mathcal T$ be a tube of rank $p \geq 1$ and $f$ be a frieze on $\mathcal T$. Then for any $i \in \Z/p\Z$ and any $n \geq 1$, we have 
				$$f(i,n)=P_n(f(i,1), \ldots, f(i+n-1,1)).$$ \hfill \qed
			\end{theorem}

			\begin{corol}\label{corol:tubes}
				A frieze on a tube is entirely determined by its values on the mouth of the tube. \hfill \qed
			\end{corol}

			\begin{rmq}
				An interesting consequence is that if $R$ is a ring and $\mathcal T$ is a tube with mouth $\mathcal M$, every function $\mathcal M \fl R$ can be completed into an $R$-frieze $\mathcal T \fl R$. This contrasts with the situation of the repetition quivers $\Z Q$ of an acyclic quiver $Q$ where if $\Sigma$ is a section in $\Z Q$, there may exist functions $\Sigma \fl R$ which do not extend to a $R$-frieze on $\Z Q$. For example, it is easily seen that if $Q$ is a quiver of Dynkin type $\mathbb A_2$, there exists no $\Z$-frieze $f$ on $\Z Q$ such that $f$ equals 2 on a section in $\Z Q$.
			\end{rmq}
	
	\subsection{Characters associated to quasi-simple modules}\label{ssection:qs}
		\begin{lem}\label{lem:shiftautomorphism}
			Let $Q$ be an acyclic quiver. Then the map $\sigma:X_M \mapsto X_{M[1]}$ induces a $\Z$-algebra automorphism of $\mathcal A(Q)$.
		\end{lem}
		\begin{proof}
			It follows from \cite{CK2} that $\mathcal A(Q)$ is the $\Z$-algebra generated by $X_M$ for $M$ rigid in $\CC_{Q}$ with the relations 
			$$X_MX_{M^*}=X_B+X_B'$$
			where $(M,M^*)$ is an exchange pair and $B,B'$ are the central terms of the corresponding exchange triangles (see \cite{BMRRT} for details). Since $[1]$ is an auto-equivalence of $\CC_{Q}$, $(M[1],M^*[1])$ is also an exchange pair and $B[1],B'[1]$ are the central terms of the corresponding exchange triangles. Thus, 
			$$X_{M[1]}X_{M^*[1]}=X_{B[1]}+X_{B'[1]}$$
			and $X_M \mapsto X_{M[1]}$ induces a $\Z$-linear homomorphism $\mathcal A(Q) \fl \mathcal A(Q)$ whose inverse is induced by $X_M \mapsto X_{M[-1]}$.
		\end{proof}

		From now on and until the end of Section \ref{section:proof}, if $Q$ denotes an euclidean quiver with radical vector $\delta$ then $e \in Q_0$ denotes a point which is either a source or a sink such that $\delta_e=1$.
		\begin{lem}\label{lem:Mqs}
			Let $Q$ be an euclidean quiver. Then for any $\lambda \in \P^1 (\k)$, there exists a unique quasi-simple module $M_{\lambda}$ in $\mathcal T_{\lambda}$ such that $M_{\lambda}(e) \simeq \k$. 
		\end{lem}
		\begin{proof}
			Let $R_0, \ldots, R_{p-1}$ be the quasi-simple modules in $\mathcal T$. It is well-known (see for instance \cite{CB:lectures}) that
			$$\delta=\sum_{i=0}^{p-1} \ddim R_i$$
			so that
			$$1 = \delta_e =\sum_{i=0}^{p-1} \dim R_i(e).$$
			Thus, there exists a unique $i \in \ens{0, \ldots, p-1}$ such that $R_i(e) \simeq \k$.
		\end{proof}
		
		\begin{lem}\label{lem:extension1inj}
			Let $Q$ be an euclidean quiver and $\lambda \in \P^1(\k)$. Assume that $e$ is a source such that $\delta_e=1$. Set $N_{\lambda} = M_{\lambda}[1]$. Then the following hold~:
			\begin{enumerate}
				\item $\dim \Ext^1_{\CC_Q}(N_\lambda, S_e) =1$~;
				\item $X_{N_{\lambda}}X_{S_e}=X_{B_\lambda}+X_{B'_\lambda}$ where $B_\lambda$ is the (unique) $\kQ$-module such that there exists a short exact sequence
				$$0 \fl N_{\lambda} \fl B_\lambda \fl S_e \fl 0$$
				and $B'_\lambda=\ker f \oplus \coker f[-1]$ for any non-zero morphism $f \in \Hom_{\CC_{Q}}(N_{\lambda},\tau S_e)$~;
				\item $B_\lambda$ and $B_\lambda'$ are transjective.
			\end{enumerate}
		\end{lem}
		\begin{proof}
			Note that $I_e \simeq S_e$ so that we have 
			\begin{align*}
				\dim \Ext^1_{\CC_{Q}}(N_{\lambda},S_e)
					&= \dim \Ext^1_{\CC_{Q}}(N_{\lambda},I_e)\\
					&= \dim \Ext^1_{\CC_{Q}}(N_{\lambda},P_e[2])\\
					&= \dim \Ext^1_{\CC_{Q}}(N_{\lambda}[-2],P_e)\\
					&= \dim \Ext^1_{\CC_{Q}}(P_e,N_{\lambda}[-2])\\
					&= \dim \Hom_{\CC_{Q}}(P_e,N_{\lambda}[-1])\\
					&= \dim \Hom_{\kQ}(P_e,\tau^{-1}N_{\lambda})\\
					&= \dim \Hom_{\kQ}(P_e,M_{\lambda})\\
					&= \dim M_{\lambda}(e)\\
					&= 1.\\
			\end{align*}
			Thus, according to \cite{CK2}, 
			$$X_{N_{\lambda}}X_{S_e}=X_{B_\lambda}+X_{B'_\lambda}$$
			where $B_\lambda,B'_\lambda$ are the (unique) objects such that there exist non-split triangles 
			$$N_{\lambda} \fl B_\lambda \fl S_e \fl N_{\lambda}[1] \textrm{ and } S_e \fl B'_\lambda \fl N_{\lambda} \fl S_e[1]$$
			in $\CC_{Q}$. Since $\Ext^1_{\kQ}(S_e,N_{\lambda}) \simeq \k$, there exists a unique $\kQ$-module $M_{\lambda}$ such that there is a non-split short exact sequence 
			$$0 \fl N_{\lambda} \fl B_\lambda \fl S_e \fl 0$$
			inducing a non-split triangle $N_{\lambda} \fl B_\lambda \fl S_e \fl N_{\lambda}[1]$ in $\CC_{Q}$. Since $B_\lambda$ is the middle term of a non-split exact sequence starting at a quasi-simple regular $\kQ$-module and ending at a preinjective (and actually injective) module, it follows that $B_\lambda$ is a preinjective $\kQ$-module so that it is identified with a transjective object in $\CC_Q$.

			Also, $\Hom_{\CC_{Q}}(N_{\lambda},\tau S_e) \simeq \Hom_{\kQ}(N_{\lambda},\tau S_e) \simeq \k$ so that a non-zero morphism $f \in \Hom_{\kQ}(N_{\lambda},\tau S_e)$ induces a triangle
			$$ \tau S_e[-1] \fl \ker f \oplus \coker f[-1] \fl N_{\lambda} \xrightarrow{f} \tau S_e$$
			in $D^b(\modd \kQ)$. Since the projection functor $D^b(\modd \kQ) \fl \CC_{Q}$ is triangulated, this gives a non-split triangle 
			$$S_e \fl B'_\lambda \fl N_{\lambda} \fl S_e[1]$$
			in $\CC_{Q}$. Since $N_\lambda$ is quasi-simple and $f$ is non-zero, then $\ker f$ is a proper submodule of $N_\lambda$ so that it is a postprojective $\kQ$-module or zero. Also, $\coker f$ is a quotient of the injective module $\tau S_e$ so that it is a preinjective $\kQ$-module or zero. Thus $\coker f[-1]$ is a transjective object in $\CC_Q$ and $B_\lambda'$ is transjective.
		\end{proof}

		\begin{lem}\label{lem:extension1proj}
			Let $Q$ be an euclidean quiver and $\lambda \in \P^1(\k)$. Assume that $e$ is a sink such that $\delta_e=1$. Set $N_{\lambda}=M_{\lambda}[-1]$. Then the following hold~:
			\begin{enumerate}
				\item $\dim \Ext^1_{\CC_{Q}}(N_{\lambda},S_e)=1$~;
				\item $X_{N_{\lambda}}X_{P_e}=X_{B_{\lambda}}+X_{B'_{\lambda}}$ where $B_{\lambda}$ is the (unique) $\kQ$-module such that there exists a short exact sequence
				$$0 \fl S_e \fl B_\lambda \fl N_{\lambda} \fl 0$$
				and $B'_{\lambda}=\ker f \oplus \coker f[-1]$ for any non-zero morphism $f \in \Hom_{\CC_{Q}}(S_e,\tau N_{\lambda})$~;
				\item $B_\lambda$ and $B_\lambda'$ are transjective.
			\end{enumerate}
		\end{lem}
		\begin{proof}
			Since $e$ is a sink, we have $S_e \simeq P_e$ so that 
			\begin{align*}
				\dim \Ext^1_{\CC_{Q}}(N_{\lambda},S_e)
					&= \dim \Ext^1_{\CC_{Q}}(N_{\lambda},P_e)\\
					&= \dim \Ext^1_{\kQ}(N_{\lambda},P_e)\\
					&= \dim \Hom_{\kQ}(P_e,\tau N_{\lambda})\\
					&= \dim \Hom_{\kQ}(P_e,M_{\lambda})\\
					&= \dim M_{\lambda}(e)\\
					&= 1
			\end{align*}
			The rest follows as for Lemma \ref{lem:extension1inj}.
		\end{proof}

	\subsection{Proof of Theorem \ref{theorem:main}}\label{ssection:mainproof}
		If order to prove Theorem \ref{theorem:main}, we need to prove that if $f$ is a cluster frieze on $\Gamma(\CC_Q)$ such that $f(P_i[1])=u_i$ for any $i \in Q_0$, then $f(M)=X_M$ for any indecomposable rigid object $M$ in $\CC_Q$. It is known that an indecomposable rigid object $M$ is either transjective or belongs to an exceptional tube.

		If $M$ is transjective, the result follows from Corollary \ref{corol:transjective}. We thus assume that $M$ belongs to a tube $\mathcal T_\lambda$ for some $\lambda \in \P^{\mathcal E}$. According to Corollary \ref{corol:tubes}, since $X_?$ and $f$ induce friezes on $\mathcal T_\lambda$, it is enough to prove that $X_?$ and $f$ coincide on the quasi-simple modules at the mouth of $\mathcal T_\lambda$. In other words, it is enough to prove that $X_{N_\lambda[k]}=f(N_\lambda[k])$ for any $k \in \ens{0, \ldots, p_\lambda-1}$. According to Lemmas \ref{lem:extension1inj} and \ref{lem:extension1proj}, we have
		$$X_{N_\lambda}=\frac{X_{B_\lambda}+X_{B_\lambda'}}{X_{S_e}}$$
		so that applying Lemma \ref{lem:shiftautomorphism}, we get 
		$$X_{N_\lambda[k]}=\frac{X_{B_\lambda[k]}+X_{B_\lambda'[k]}}{X_{S_e[k]}}$$
		for any $k \in \ens{0, \ldots, p_\lambda-1}$.

		Since $S_e[k], B_\lambda[k]$ and $B_\lambda'[k]$ are transjective, we have $X_{S_e[k]}=f(S_e[k])$, $X_{B_\lambda[k]}=f(B_\lambda[k])$ and $X_{B_\lambda'[k]}=f(B_\lambda'[k])$ for any $k \in \ens{0, \ldots, p_\lambda-1}$. Thus, since $f$ is a cluster frieze, then for any $k \in \ens{0, \ldots, p_\lambda-1}$, we get
		\begin{align*}
			f(N_{\lambda}[k])
				&=\frac{f(B_\lambda[k])+f(B_\lambda'[k])}{f(S_e[k])}\\
				&=\frac{X_{B_\lambda[k]}+X_{B_\lambda'[k]}}{X_{S_e[k]}}\\
				&=X_{N_{\lambda}[k]}
		\end{align*}
		and Theorem \ref{theorem:main} is proved. \hfill \qed

	\subsection{Proof of Theorem \ref{theorem:strong}}\label{ssection:strongroof}
		If order to prove Theorem \ref{theorem:main}, we need to prove that if $f$ is a strong cluster frieze on $\Gamma(\CC_Q)$ such that $f(P_i[1])=u_i$ for any $i \in Q_0$, then $f(M)=X_M$ for any indecomposable object $M$ in $\CC_Q$. 

		Since $f$ is a strong cluster frieze, it is in particular a cluster frieze and it thus follows that $f(M)=X_M$ for any object $M$ in the transjective component or in an exceptional tube of $\Gamma(\CC_Q)$. It is thus sufficient to prove that $f$ and $X_?$ coincide on homogeneous tubes. According to Corollary \ref{corol:tubes}, since $X_?$ and $f$ induce friezes on tubes, it is enough to prove that $X_?$ and $f$ coincide on the mouths of homogeneous tubes. Fix thus $\lambda \in \P^{\mathcal H}$. It follows from Lemmas \ref{lem:extension1inj} and \ref{lem:extension1proj} that
		$$X_{N_\lambda}=\frac{X_{B_\lambda}+X_{B_\lambda'}}{X_{S_e}}.$$
		As before, since $S_e, B_\lambda$ and $B_\lambda'$ are transjective, we have $X_{S_e}=f(S_e)$, $X_{B_\lambda}=f(B_\lambda)$ and $X_{B_\lambda'}=f(B_\lambda')$. Moreover, since $f$ is a strong cluster frieze, we have 
		\begin{align*}
			f(N_{\lambda})
				&=\frac{f(B_\lambda)+f(B_\lambda')}{f(S_e)}\\
				&=\frac{X_{B_\lambda}+X_{B_\lambda'}}{X_{S_e}}\\
				&=X_{N_{\lambda}}
		\end{align*}
		which proves Theorem \ref{theorem:strong}. \hfill \qed

\section{Euler characteristics of complete quiver grassmannians}\label{section:GrM}
	Given an acyclic quiver $Q$, we construct $\Z$-friezes on $\Gamma(\CC_Q)$ and see that they allow to compute algorithmically the Euler characteristics of the following projective varieties.

	\begin{defi}
		Let $Q$ be an acyclic quiver and $M$ be a $\kQ$-module. The \emph{complete quiver grassmannian} is 
		$$\Gr(M)=\bigsqcup_{\textbf e \in \Z_{\geq 0}^{Q_0}}\Gr_{\textbf e}(M).$$
	\end{defi}
	Since $\Gr_{\textbf e}(M)$ is empty for all but finitely many $\textbf e \in \Z_{\geq 0}^{Q_0}$, then $\Gr(M)$ is a finite union of projective varieties and is thus a projective variety. We denote by $\chi(\Gr(M))$ its Euler characteristic.
	
	Let $\ev_1$ be the ring homomorphism
	$$\ev_1: \left\{\begin{array}{rcll}
		\Z[\textbf u^{\pm 1}] & \fl & \Z, \\
		u_i & \mapsto & 1 & \textrm{ for any }i \in Q_0.
	\end{array}\right.$$

	\begin{lem}\label{lem:GrMev}
		Let $Q$ be an acyclic quiver, then for any $\kQ$-module $M$, 
		$$\ev_1(X_M)=\chi(\Gr(M)).$$
	\end{lem}
	\begin{proof}
		Let $M$ be a $\kQ$-module. Then
		\begin{align*}
			\ev_1(X_M) 
				&= \sum_{\textbf e \in \Z_{\geq 0}^{Q_0}} \chi(\Gr_{\textbf e}(M)) \ev_1(u_i^{-\<\textbf e,S_i\>-\<S_i, \ddim M- \textbf e\>})\\
				&= \sum_{\textbf e \in \Z_{\geq 0}^{Q_0}} \chi(\Gr_{\textbf e}(M))\\
				&= \chi(\bigsqcup_{\textbf e \in \Z_{\geq 0}^{Q_0}} \Gr_{\textbf e}(M)) \\
				&= \chi(\Gr(M)).
		\end{align*}
	\end{proof}

	\begin{corol}\label{corol:GrMtransjective}
		Let $Q$ be an acyclic quiver and $f$ be a $\Z$-frieze on $\Gamma(\CC_Q)$ such that $f(P_i[1])=1$ for any $i \in Q_0$. Then $f(M)=\chi(\Gr(M))$ for any transjective object $M$. 

		In particular, $\chi(\Gr(M))>0$ for any postprojective or preinjective $\kQ$-module $M$.
	\end{corol}
	\begin{proof}
		It follows from Proposition \ref{prop:Xfrieze} that $X_?$ induces a frieze on $\Gamma(\CC_Q)$. Since $\ev_1:\Z[\textbf u^{\pm 1}] \fl \Z$ is a ring homomorphism, it follows that $\ev_1 \circ X_?$ is a $\Z$-frieze on $\Gamma(\CC_Q)$ such that $(\ev_1 \circ X_?)(P_i[1])=1$ for any $i \in Q_0$. Thus, $f$ and $\ev_1 \circ X_?$ coincide on a section in the transjective component $\mathcal P$. By Lemma \ref{lem:asection} the friezes coincide on $\mathcal P$. It thus follows from Lemma \ref{lem:GrMev} that $f(M)=\chi(\Gr(M))$ for any transjective object $M$. An easy induction on the transjective component proves that $f(M) >0$ for any transjective object $M$.
	\end{proof}

	\begin{corol}\label{corol:GrMrigid}
		Let $\CC_{Q}$ be the cluster category of an euclidean quiver $Q$ and let $f$ be a cluster $\Z$-frieze on $\Gamma(\CC_{Q})$ such that $f(P_i[1])=1$ for any $i \in Q_0$. Then $f(M)=\chi(\Gr(M))$ for any rigid $\kQ$-module $M$.
	\end{corol}
	\begin{proof}
		As above, $f$ and $M \mapsto \chi(\Gr(M))$ coincide on the transjective component. Now, as for Theorem \ref{theorem:main}, it is easy to check that if $f$ is a $\Z$-frieze on $\Gamma(\CC_Q)$ such that $f(P_i[1])=1$ for any $i \in Q_0$, then $f$ coincides with $\ev_1 \circ X_?$ on every object contained either in the transjective component or in an exceptional tube. In particular, $f(M)=\ev_1(X_M)=\chi(\Gr(M))$ for any rigid $\kQ$-module $M$. 
	\end{proof}

	Similarly, we obtain~:
	\begin{corol}\label{corol:GrM}
		Let $\CC_{Q}$ be the cluster category of an euclidean quiver $Q$. Let $f$ be a strong cluster $\Z$-frieze on $\Gamma(\CC_{Q})$ such that $f(P_i[1])=1$ for any $i \in Q_0$. Then $f(M)=\chi(\Gr(M))$ for any $\kQ$-module $M$. \hfill \qed
	\end{corol}

\section{Explicit determination of $B_{\lambda}$ and $B_{\lambda}'$}\label{section:explicit}
	Let $Q$ be an euclidean quiver. Once defined, the cluster frieze provides an efficient tool for computations in the cluster algebra $\mathcal A(Q)$. Nevertheless, the definition of a cluster frieze on $\Gamma(\CC_Q)$ requires the determination of the pair $(B_\lambda,B_\lambda')$ of objects arising as middle terms of triangles with extreme terms $N_\lambda$ and $S_e$. This determination is the only obstruction for getting a complete algorithm for computing cluster variables. The aim of this section is to explicit these objects for arbitrary euclidean quivers. This is based on a case-by-case analysis.

	\subsection{Homogeneous tubes}
		We first prove that for any euclidean quiver $Q$ and any parameter $\lambda \in \P^{\mathcal H}$ of an homogeneous tube, the transjective objects $B_\lambda$ and $B_\lambda'$ do not depend on the choice of $\lambda$.
		\begin{lem}\label{lem:BHproj}
			Let $Q$ be an euclidean quiver. Assume that $e$ is a sink. Then, for any $\lambda \in \P^{\mathcal H}$, the following hold~:
			\begin{enumerate}
				\item $B_\lambda$ is the unique (up to isomorphism) indecomposable module of dimension vector $\delta + \ddim S_e$~;
				\item $B_\lambda' \simeq C[-1]$ where $C$ is the unique (up to isomorphism) indecomposable module of dimension vector $\delta - \ddim S_e$.
			\end{enumerate}
		\end{lem}
		\begin{proof}
			According to Lemma \ref{lem:extension1proj}, $B_\lambda$ is the only module such that there exists a non-split exact sequence 
			$$0 \fl S_e \fl B_\lambda \fl N_\lambda \fl 0$$
			and $B_\lambda$ is postprojective. 

			We denote by $\d_X=\<\delta, \ddim X\>$ the \emph{defect} of a representation $X$ of $Q$ (see e.g. \cite{DR:memoirs} for details). Since $\delta_e=1$, we have $\d_{S_e}=\d_{P_e}=-1$. Applying this form to the short exact sequence above, we get $\d_{B_\lambda}=\d_{S_e}=-1$ so that $B_\lambda$ is an indecomposable postprojective module. Up to isomorphism, it is thus determined by its dimension vector which is $\delta + \ddim S_e$.

			Now, according to Lemma \ref{lem:extension1proj}, $B_\lambda'$ is isomorphic to $\ker f \oplus \coker f[-1]$ for any non-zero morphism $f \in \Hom_{\kQ}(S_e,\tau N_\lambda)$. Since $S_e$ is simple, then $\ker f$ is zero and thus $B_\lambda' \simeq C[-1]$ where $C=\coker f$. Applying the defect form to the short exact sequence
			$$0 \fl S_e \xrightarrow{f} \tau N_\lambda \fl C \fl 0$$
			we get $\d_{C}=-\d_{S_e}=1$. Since $C$ is preinjective, it is indecomposable and thus entirely determined by its dimension vector which is $\delta -\ddim S_e$.
		\end{proof}

		We put a partial order $\leq $ on $\Z^{Q_0}$ by setting $(e_i)_{i \in Q_0} \leq (f_i)_{i \in Q_0}$ if $e_i \leq f_i$ for every $i \in Q_0$.

		\begin{lem}\label{lem:BHinj}
			Let $Q$ be an euclidean quiver. Assume that $e$ is a source. Then, for any $\lambda \in \P^{\mathcal H}$, the following hold~:
			\begin{enumerate}
				\item $B_\lambda$ is the unique (up to isomorphism) indecomposable module of dimension vector $\delta + \ddim S_e$~;
				\item if $\ddim (\tau S_e) \leq \delta$, then $B_\lambda'$ is the unique (up to isomorphism) indecomposable module of dimension vector $\delta - \ddim (\tau S_e)$~;
					\item if $\delta \leq \ddim (\tau S_e)$, then $B_\lambda' \simeq C[-1]$ where $C$ is the unique (up to isomorphism) indecomposable module of dimension vector $\ddim (\tau S_e) - \delta$.
				\end{enumerate}
			\end{lem}
			\begin{proof}
				According to Lemma \ref{lem:extension1inj}, $B_\lambda$ is the only module such that there exists a non-split exact sequence 
				$$0 \fl N_\lambda \fl B_\lambda \fl S_e \fl 0.$$
				Since $N_\lambda$ is quasi-simple and $S_e$ is injective, then $B_\lambda$ is preinjective. Applying the defect form, we get $\d_{B_\lambda}=\d_{S_e}=1$ so that $B_\lambda$ is indecomposable preinjective. Up to isomorphism, it is thus determined by its dimension vector which is $\delta + \ddim S_e$.

				Now, according to Lemma \ref{lem:extension1inj}, $B_\lambda'$ is isomorphic to $\ker f \oplus \coker f[-1]$ for any non-zero morphism $f \in \Hom_{\kQ}(N_\lambda,\tau S_e)$. We set $K = \ker f$ and $C = \coker f$. Since $N_\lambda$ is quasi-simple, then $K$ is necessarily postprojective and $C$ is necessarily preinjective. Applying the defect form to the exact sequence
				$$0 \fl K \fl N_\lambda \xrightarrow{f} \tau S_e \fl C \fl 0$$
				we get $\d_K=1+\d_C$. Since $\d_K \leq 0$ and $\d_C \geq 0$, we get that either $K=0$ and $C$ is indecomposable or $C=0$ and $K$ is indecomposable.

				If $C=0$, then $f$ is an epimorphism and thus $\ddim (\tau S_e) \leq \delta$. In this case, $B_\lambda'=K$ is entirely determined by its dimension vector which is $\delta - \ddim (\tau S_e)$. 

				If $K=0$, then $f$ is a monomorphism so that $\delta \leq \ddim (\tau S_e)$. In this case $C$ is entirely determined by its dimension vector which is $\ddim (\tau S_e) - \delta$ and $B_\lambda' = C[-1]$.
			\end{proof}

			\begin{corol}
				Let $Q$ be an euclidean quiver. Then for any $\lambda, \mu \in \P^{\mathcal H}$, we have 
				$$B_\lambda \simeq B_\mu \textrm{ and } B_{\lambda}' \simeq B_\mu'.$$ \hfill \qed
			\end{corol}

	\subsection{Exceptional tubes in type $\Aaffine$}\label{ssection:Aaffine}
		Let $Q$ be an euclidean quiver of type $\Aaffine_{r,s}$, that is, with $r$ arrows going clockwise and $s$ arrows going counter-clockwise. The Auslander-Reiten-quiver $\Gamma(\CC_{Q})$ contains two exceptional tubes $\mathcal T_0$ and $\mathcal T_1$ of respective ranks $r$ and $s$. 

		Since $Q$ is acyclic it contains at least one sink and at least one source. Moreover, since $\delta_i=1$,  for any $i \in Q_0$, in this case, the point $e$ can be any of these sinks or sources. In order to fix notations, we consider the case where $e$ is a sink. 

		There exists thus a unique clockwise (or counter-clockwise, respectively) arrow in $Q_1$ with target $e$. We denote by $i_+$ (or $i_-$, respectively) the source of this unique clockwise (or counter-clockwise, respectively) arrow. Locally, $Q$ can be depicted as follows~:
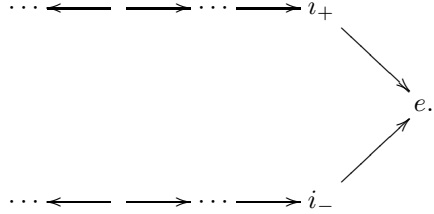
		$$\xymatrix{
			\cdots & \ar[l]  \ar[r] & \cdots \ar[r] & i_+ \ar[rd] \\
			& & & & e. \\
			\cdots & \ar[l]  \ar[r] & \cdots \ar[r] & i_- \ar[ru] 
		}$$
		We denote by $M_0$ (or $M_1$) the quasi-simple module in $\mathcal T_0$ (or $\mathcal T_1$, respectively) given by Lemma \ref{lem:Mqs}. 

		It is not hard to prove (see for instance \cite[\S 9.6]{mathese}) that $M_0$ and $M_1$ are given by~:
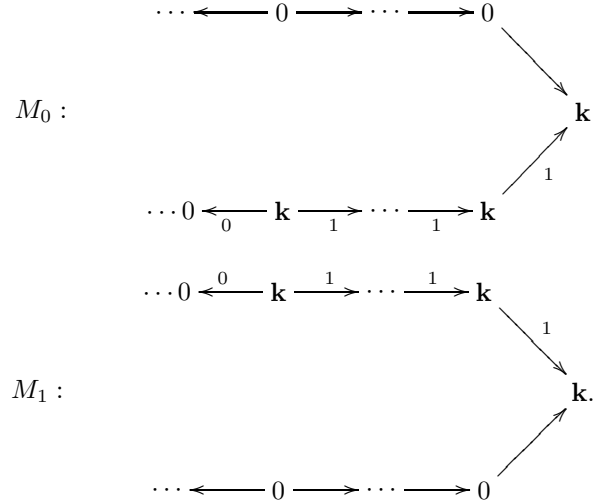
		$$\xymatrix{
			& \cdots & \ar[l] 0 \ar[r] & \cdots \ar[r] & 0 \ar[rd] \\
			M_0: & & & & & \k \\
			& \cdots 0 & \ar[l]^0 \k \ar[r]_1 & \cdots \ar[r]_1 & \k \ar[ru]_1 
		}$$

		$$\xymatrix{
			& \cdots 0 & \ar[l]_0 \k \ar[r]^1 & \cdots \ar[r]^1 & \k \ar[rd]^1 \\
			M_1: & & & & & \k. \\
			& \cdots & \ar[l] 0 \ar[r] & \cdots \ar[r] & 0 \ar[ru] 
		}$$
		
		The set $\ens{\tau^{k}M_0|0 \leq k \leq r-1}$ is a complete list of pairwise non-isomorphic quasi-simple modules in $\mathcal T_0$ and the set $\ens{\tau^{k}M_1|0 \leq k \leq s-1}$ is a complete list of pairwise non-isomorphic quasi-simple modules in $\mathcal T_1$. 
		Since we considered a sink $e$, we set $N_0=M_0[-1]$ and $N_1=M_1[-1]$.  We easily check that $N_0 \simeq P_{i+}/S_e$. We have a non-split extension of $\kQ$-modules
		$$0 \fl S_e \fl P_{i_+} \fl N_0 \fl 0.$$
		Moreover, if $f$ is a non-zero morphism $S_e \fl \tau N_0 \simeq M_0$, then $\ker f=0$ since $S_e$ is simple and $\coker f\simeq M_0/S_e=I_{i_-}$. Thus, we get 
		$$B_0 \simeq P_{i_+} \textrm{ and } B_0' \simeq I_{i_-}[-1] \simeq P_{i_-}[1].$$
		Similarly, 
		$$B_1 \simeq P_{i_-} \textrm{ and } B_1' \simeq I_{i_+}[-1] \simeq P_{i_+}[1].$$

		For $\lambda \in \P^{\mathcal H}$, it follows from Lemma \ref{lem:BHproj} that $B_\lambda$ is the unique indecomposable postprojective module $E$ with dimension vector $\delta + \ddim S_e$ and $B_\lambda' \simeq C[-1]$ where $C$ is the unique indecomposable preinjective module with dimension vector $\delta - \ddim S_e$.

 		Summing up in a table, we get~:
		\begin{table}[H]
			$$\begin{array}{|c|c|c|c|c|c|c|}
				\hline
				\lambda \in \P^1(\k) & p_{\lambda} & B_{\lambda } & B'_{\lambda }\\
				\hline
				0 & r & P_{i_+} & P_{i_-}[1]\\
				\hline
				1 & s & P_{i_-} & P_{i_+}[1]\\
				\hline
				\lambda \in \P^{\mathcal H} & 1 & E & C[-1]\\
				\hline
			\end{array}$$
			\caption{$B_\lambda$ and $B_\lambda'$ in type $\Aaffine_{r,s}$}\label{tab:Aaffine}
		\end{table}

	\subsection{Exceptional tubes in type $\Daffine$}\label{ssection:Daffine}
		We now assume that $Q$ is an euclidean quiver of type $\Daffine_{n+3}$ with $n \geq 1$. We moreover assume that $Q$ is equipped with the orientation in the tables of \cite{DR:memoirs}. That is
		$$
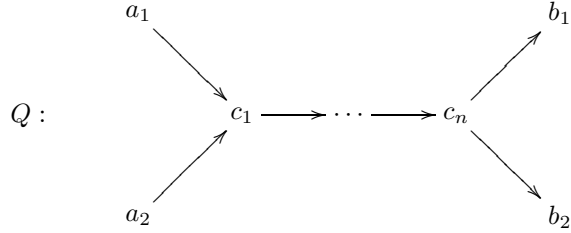
\xymatrix{
			&	a_1 \ar[rd] &&&& b_1 \\
			Q:	 && c_1 \ar[r] &\cdots \ar[r] & c_n \ar[rd]\ar[ru] \\ 
			&	a_2 \ar[ru] &&&& b_2 \\
		}$$
		We set $e=b_1$ which is a sink.

		There are three exceptional tubes $\mathcal T_{1}, \mathcal T_{\infty}$ and $\mathcal T_0$ of respective ranks $n+1$, 2 and 2. We denote by $M_1$, $M_{\infty}$ and $M_0$ the quasi-simple modules given by Lemma \ref{lem:Mqs}.

		\subsubsection{The tube of rank $n+1$}
			The quasi-simple modules in $\mathcal T_1$ are the $S_{c_i}, i=1, \ldots, n$ and the sincere module $M_{1}$ such that $M_1(i) \simeq \k$ for every $i \in Q_0$ and $M_1(\alpha)=1_{\k}$ for every $\alpha \in Q_1$. We have $\tau S_{c_i} \simeq S_{c_{i+1}}$ if $i \neq n$ and $\tau S_{c_n} \simeq M_1$. 

			We set $N_1=M_1[-1] \simeq S_{c_n}$. With the notations of Lemma \ref{lem:extension1proj}, we get 
			$$B_1 \simeq P_{c_n}/P_{b_2}$$
			and $B'_1=\ker f \oplus \coker f[-1]$ for $0 \neq f \in \Hom_{\kQ}(S_{b_1},M_1)$. But $\ker f= 0$ since $P_{b_1} \simeq $ is simple to that 
			$$B'_1=\coker f[-1] \simeq I_{b_2}[-1] \simeq P_{b_2}[1].$$

		\subsubsection{Tubes of rank $2$}
			Let $\mathcal T_0$ be the tube whose quasi simples are 
			$$
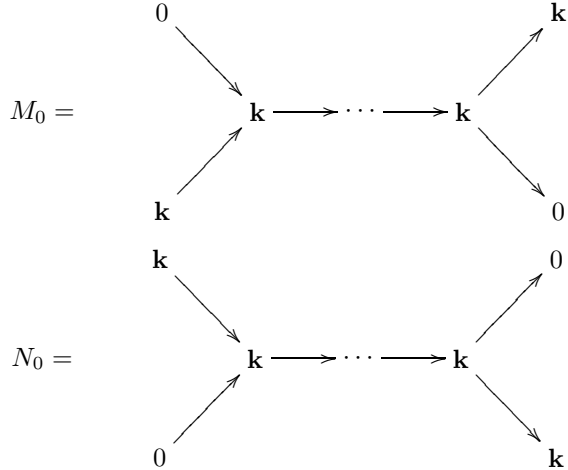
\xymatrix{
				&	0 \ar[rd] &&&& \k \\
				M_0=	 && \k \ar[r] &\cdots \ar[r] & \k \ar[rd]\ar[ru] \\ 
				&	\k \ar[ru] &&&& 0
			}$$
			$$\xymatrix{
				&	\k \ar[rd] &&&& 0 \\
				N_0=	 && \k \ar[r] &\cdots \ar[r] & \k \ar[rd]\ar[ru] \\ 
				&	0 \ar[ru] &&&& \k
			}$$
			With the notations of Lemma \ref{lem:extension1proj}, we get 
			$$B_0 \simeq P_{a_1}$$
			and $B_0'=\ker f \oplus \coker f[-1]$ for $0 \neq f \in \Hom_{\kQ}(S_{b_1}, M_0)$. Since $S_{b_1}$ is simple, $\ker f=0$ and thus $\coker f$ is the representation 
			$$\xymatrix{
				0 \ar[rd] &&&& 0 \\
				& \k \ar[r] &\cdots \ar[r] & \k \ar[rd]\ar[ru] \\ 
				\k \ar[ru] &&&& 0.
			}$$
			In particular, $\coker f$ is indecomposable preinjective so that it is determined by its dimension vector. An easy induction on $n$ shows that $\ddim \coker f = \ddim \tau^{n}I_{a_2}$ so that $\coker f \simeq \tau^n I_{a_2} \simeq I_{a_2}[n]$ and thus 
			$$B'_0 \simeq I_{a_2}[n-1].$$

			For $\mathcal T_{\infty}$, we notice that the group $\Z/2\Z$ acts by automorphisms on $\CC_{Q}$ by exchanging points $a_1$ and $a_2$. This action preserves the point $e=b_1$ so that 
			$$B_\infty \simeq P_{a_2} \textrm{ and } B_\infty' \simeq I_{a_1}[n-1].$$
			
			For $\lambda \in \P^{\mathcal H}$, it follows from Lemma \ref{lem:BHproj} that $B_\lambda$ is the unique indecomposable postprojective module $E$ with dimension vector $\delta + \ddim S_e$ and $B_\lambda' \simeq C[-1]$ where $C$ is the unique indecomposable preinjective module with dimension vector $\delta - \ddim S_e$.

			Summing up in a table, we get~:
			\begin{table}[H]
				$$\begin{array}{|c|c|c|c|c|c|c|}
					\hline
					\lambda \in \P^1(\k) & p_{\lambda} & B_{\lambda } & B'_{\lambda } & e\\
					\hline
					1 & n+1 & P_{c_n}/P_{b_2} & I_{b_2}[-1] & e=b_1\\
					\hline
					0 & 2 & P_{a_1} & I_{a_2}[n-1] & e=b_1\\
					\hline
					\infty & 2 & P_{a_2} & I_{a_1}[n-1] & e=b_1\\
					\hline
					\lambda \in \P^{\mathcal H} & 1 & E & C[-1] & e=b_1\\
					\hline
				\end{array}$$
				\caption{$B_\lambda$ and $B_\lambda'$ in type $\Daffine_{n+3}, n \geq 1$}\label{tab:Daffine}
			\end{table}

	\subsection{Exceptional tubes in type $\Eaffine_6$}\label{ssection:E6affine}
		Let $Q$ be the quiver of euclidean type $\Eaffine_6$ with the orientation considered in the tables of \cite{DR:memoirs}.
		$$\xymatrix{
			&&& 5 \ar[d] \\
			&&& 4 \ar[d] \\
		Q: & 3 \ar[r] & 2 \ar[r] & 1 & \ar[l] 6 & \ar[l] 7
		}$$
		
		There are three exceptional tubes $\mathcal T_{\infty},\mathcal T_0,\mathcal T_1$ of respective ranks 2, 3 and 3. We set $e=7$ which is a source. Note that $I_e \simeq S_e = S_7$ and that $\tau S_e \simeq S_6$.

		For any $\lambda \in \P^{\mathcal E}$, we denote by $M_{\lambda}$ the quasi-simple module given in Lemma \ref{lem:Mqs} and set $N_{\lambda} \simeq M_{\lambda}[1]$. We denote by $B_{\lambda}$ and $B_{\lambda}'$ the unique objects in $\CC_{Q}$, given by Lemma \ref{lem:extension1inj} such that there is a non-split short exact sequence 
		$$0 \fl N_{\lambda} \fl B_{\lambda} \fl S_e \fl 0$$
		and $B_{\lambda}'=\ker f \oplus \coker f[-1]$ for some non-zero morphism $f \in \Hom_{\kQ}(N_{\lambda},\tau S_e)$. Since $\tau S_e$ is simple, $B_\lambda' =\ker f$. As in the proof of Lemma \ref{lem:BHinj}, it is easy to see that $B_{\lambda}$ and $B_{\lambda}'$ are determined by their dimension vectors. 

		The following table sums up the dimension vectors for $\mathcal T_0$ and $\mathcal T_{\infty}$~:
		$$\begin{array}{|c|c|c|c|}
		\hline 
			\lambda \in \P^1(\k) & \ddim N_{\lambda} & \ddim B_{\lambda} & \ddim B_{\lambda}' \\
		\hline
			{\infty} & [1101010] & [1101011] & [1101000]\\
			{0} & [1110010] & [1101011] & [1110000]\\
		\hline
		\end{array}$$

 		Now it is possible, using for example the algorithms provided in Section \ref{ssection:recognise}, to compute explicitly the indecomposable modules having these dimension vectors. We get
		$$B_{\infty} \simeq \tau^3 I_7 \textrm{ and } B_\infty' \simeq \tau^{-1} P_7$$
		and
		$$B_{0} \simeq \tau^2 I_5 \textrm{ and } B_0' \simeq P_3.$$
		For $\mathcal T_1$, we deduce it from $\mathcal T_0$ using the fact that the action of the product of transpositions $(2,6)(3,7)$ on $Q_0$ induces an action on $\CC_{Q}$ by automorphisms. Since this action exchanges 7 and 3, in this case, we have to take $e=3$ and we get 
		$$B_1 \simeq \tau^2 I_5 \textrm{ and } B_1' \simeq P_7.$$

		For $\lambda \in \P^{\mathcal H}$ and $e=7$, we can apply Lemma \ref{lem:BHinj} and see that $B_\lambda$ is the unique indecomposable (preinjective) module with dimension vector $[3212122]$ and $B_\lambda'$ is the unique indecomposable (preprojective) module with dimension vector $[3212111]$. Applying the recognition algorithm given in Section \ref{ssection:recognise}, we get 
		$$B_\lambda \simeq \tau^6 I_7 \textrm{ and } B_\lambda' \simeq \tau^{-4} P_7.$$

		Summing up in a table, we get
		\begin{table}[H]
			$$\begin{array}{|c|c|c|c|c|c|c|}
				\hline
				\lambda \in \P^1(\k) & p_{\lambda} & B_{\lambda } & B'_{\lambda } & e\\
				\hline
				0 & 2 & \tau^3 I_7 & \tau^{-1} P_7 & e=7\\
				\hline
				1 & 3 & \tau^2 I_5 & P_3 & e=7\\
				\hline
				\infty & 3 & \tau^2 I_5 & P_7 & e=3\\
				\hline
				\lambda \in \P^{\mathcal H} & 1 & \tau^7 I_7 & \tau^{-4} P_7 & e=7\\
				\hline
			\end{array}$$
			\caption{$B_\lambda$ and $B_\lambda'$ in type $\Eaffine_{6}$}\label{tab:E6affine}
		\end{table}
		
		Note that an explicit computations of the corresponding variables is given in Section \ref{section:varE6}.

	\subsection{Exceptional tubes in type $\Eaffine_7$}\label{ssection:E7affine}
		Let $Q$ be the quiver of euclidean type $\Eaffine_7$ with the orientation considered in \cite{DR:memoirs}.
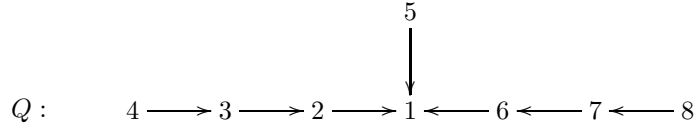
		$$\xymatrix{
				&&&& 5 \ar[d] \\
		Q: & 4 \ar[r] & 3 \ar[r] & 2 \ar[r] & 1 & \ar[l] 6 & \ar[l] 7 & \ar[l] 8
		}$$
		There are three exceptional tubes $\mathcal T_\infty,\mathcal T_0,\mathcal T_1$ of respective ranks 2,3 and 4. We consider $e=8$ which is a source. We have $I_e \simeq S_e =S_8$ and $\tau S_e \simeq S_7$. As for type $\Eaffine_6$, we get the following table~:

		\begin{table}[H]
			$$\begin{array}{|c|c|c|c|c|c|c|}
				\hline
				\lambda \in \P^1(\k) & p_{\lambda} & B_{\lambda } & B'_{\lambda } & e\\
				\hline
				{\infty} & 2 & \tau^6 I_4 & \tau^{-4} P_4 & e=8 \\
				\hline
				{0} & 3 & \tau^4 I_8 & \tau^{-2} P_8 & e=8 \\
				\hline
				{1} & 4 & \tau^3 I_4 & \tau^{-1} P_4 & e=8 \\
				\hline
				\lambda \in \P^{\mathcal H} & 1 & \tau^{12} I_8 & \tau^{-10} P_8 & e=8 \\
				\hline
			\end{array}$$
			\caption{$B_\lambda$ and $B_\lambda'$ in type $\Eaffine_{7}$}\label{tab:E7affine}
		\end{table}

	\subsection{Exceptional tubes in type $\Eaffine_8$}\label{ssection:E8affine}
		Let $Q$ be the quiver of euclidean type $\Eaffine_8$ with the canonical orientation
		$$
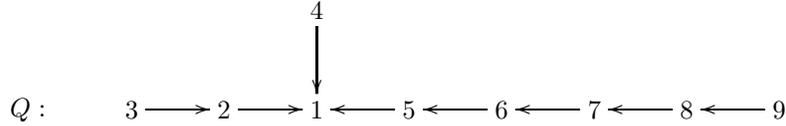
\xymatrix{
				&&& 4 \ar[d] \\
		Q: & 3 \ar[r] & 2 \ar[r] & 1 & \ar[l] 5 & \ar[l] 6 & \ar[l] 7 & \ar[l] 8& \ar[l] 9
		}$$
		There are three exceptional tubes $\mathcal T_\infty,\mathcal T_0,\mathcal T_1$ of respective ranks 2,3 and 5. We take $e=9$ which is a source. We have $I_e \simeq S_e=S_9$ and $\tau S_e \simeq S_8$. As before, we get~:
		\begin{table}[H]
			$$\begin{array}{|c|c|c|c|c|c|}
				\hline
				\lambda \in \P^1(\k) & p_{\lambda} & B_{\lambda } & B'_{\lambda }\\
				\hline
				{\infty} & 2 & \tau^{15} I_9 & \tau^{-13} P_9\\
				\hline
				{0} & 3 & \tau^{10} I_9 & \tau^{-8} P_9\\
				\hline
				{1} & 5 & \tau^{6} I_9 & \tau^{-4} P_9\\
				\hline
				\lambda \in \P^{\mathcal H} & 1 & \tau^{30} I_9 & \tau^{-28} P_9\\
				\hline
			\end{array}$$
			\caption{$B_\lambda$ and $B_\lambda'$ in type $\Eaffine_{8}$}\label{tab:E8affine}
		\end{table}

	\subsection{Changing orientations}\label{ssection:orientations}
		For simplicity, we worked with prescribed orientations, except for $\Aaffine$. Nevertheless, our techniques can actually be adapted to any orientation. 

		A first method would consist in doing again the computations made earlier. Indeed, we worked with a fixed orientation in order to use the tables provided in \cite{DR:memoirs} but there is no theoretical obstruction for doing all the computations with a different orientation.

		A second method consists in observing that the strong isomorphism of cluster algebras (in the sense of \cite{cluster2}) corresponding to the orientation change can be realised combinatorially. We detail this method here.
		
		Let $Q'$ be an euclidean quiver equipped with an arbitrary orientation. The cluster algebra $\mathcal A(Q',\textbf x)$ contains a seed $(Q,\textbf u)$ such that $Q$ is an euclidean quiver of the same euclidean type as $Q'$ equipped with one of the orientations considered above. The cluster categories $\CC_Q$ and $\CC_{Q'}$ are equivalent and in particular, have the same Auslander-Reiten quiver $\Gamma$. 

		We denote by $P_i'$, with $i \in Q'_0$, the indecomposable projective $\kQ'$-modules and by $P_i$, with $i \in Q_0$, the indecomposable projective $\kQ$-modules. We denote by $X_?:\Ob(\CC_Q) \fl \Z[\textbf u^{\pm 1}]$ (or $X'_?:\Ob(\CC_{Q'}) \fl \Z[\textbf x^{\pm 1}]$ the cluster character associated to $Q$ (or ${Q'}$, respectively) and we set $f:\Gamma \fl \Z[\textbf u^{\pm 1}]$ the cluster frieze on $\Gamma$ sending $P_i[1]$ to $u_i$ for every $i \in Q_0$. 

		Let $\phi: \mathcal A(Q,\textbf u) \fl \mathcal A(Q',\textbf x)$ be the strong isomorphism of cluster algebras. According to Theorem \ref{theorem:main}, the cluster variables in $\mathcal A(Q,\textbf u)$ are the $f(M)$ where $M$ runs over the points of $\Gamma$ corresponding to indecomposable rigid objects in $\CC_{Q}$. Let $v$ be such a point in $\Gamma$. Then $v$ corresponds to some indecomposable object $M$ in $\CC_{Q}$ and to $\Phi(M)$ in $\CC_{Q'}$. It follows from \cite[\S 5]{Palu} that
		$$X'_{M}=\phi(X_{\Phi(M)})=\phi(f(\Phi(M))).$$

		Now, in order to compute cluster variables in $\mathcal A(Q',\textbf x)$, it is enough to compute first the cluster frieze $f$ on $\Gamma$ and then to replace each $u_i$ by its Laurent expansion in the cluster $\textbf x$. The first part can be done algorithmically using Theorem \ref{theorem:main} and in order to find the expansion of $u_i$ in $\textbf x$, one simply constructs the frieze $f'$ on $\mathcal P$ such that $f'(P_i'[1])=x_i$ for any $i \in Q_0$ and then replace $u_i$ by $f'(P_i[1])$ in the values of the frieze $f$. 

\section{Algorithms}\label{section:algorithms}
	\subsection{Recognising indecomposable transjective modules}\label{ssection:recognise}
		We now detail the algorithms used in Section \ref{section:explicit} in order to recognise indecomposable modules in the transjective component from their dimension vectors. The algorithm is elementary and actually works for any acyclic quiver. 

		Let $Q$ be an acyclic quiver. We assume that we know explicitly the dimension vectors of the indecomposable projective $\kQ$-modules and of the indecomposable injective $\kQ$-modules (this can be achieved with an elementary computer program).

		Let $C$ be the \emph{Cartan matrix} of $\kQ$, that is the matrix whose $i$-th column is $\ddim P_i$ or, equivalently whose $i$-th line is $\ddim I_i$ (see for instance \cite[\S III]{ASS}). Let $\Phi=-C^t C^{-1}$ be the \emph{Coxeter matrix} of $\kQ$. Note that if $Q$ is an euclidean quiver equipped with one of the orientations provided in Section \ref{section:explicit}, the explicit Coxeter matrix can be found in \cite{SS:volume2}.
		
		Then for any $i \in Q_0$, $n \geq 0$, we have 
		$$\ddim \tau^{-n} P_i = \Phi^{-n} \ddim P_i~;$$
		$$\ddim \tau^{n} I_i = \Phi^{n} \ddim P_i.$$

		Thus, if $M$ is an indecomposable postprojective module whose dimension vector is known, one can apply the following procedure. 
		\begin{algorithm}[H]
			\caption{An algorithm for recognising indecomposable postprojective modules}\label{algo:recpostprojective}
			\begin{algorithmic}[1]
			\REQUIRE $M$ is indecomposable postprojective
			\STATE  $n=0$
			\WHILE {$f=0$}
			\FOR {$i \in Q_0$}
			\IF {$\ddim M=\Phi^{-n} \ddim P_i$}
			\STATE $R:=\tau^{-n}P_i$
			\STATE $f:=1$
			\ENDIF
			\ENDFOR
			\STATE $n:=n+1$
			\ENDWHILE
			\RETURN $R$ // The module $M$ is thus isomorphic to $R$.
			\end{algorithmic}
		\end{algorithm}
		The algorithm is dual for indecomposable preinjective modules.
		

	\subsection{Computing cluster characters in the euclidean case}\label{ssection:compute}
		Let $Q$ be an euclidean quiver. In this section, we give algorithms for computing cluster characters associated to arbitrary indecomposable $\kQ$-modules. As usual, for any $\lambda \in \P^1(\k)$, $M_\lambda$ is the quasi-simple module in the tube $\mathcal T_\lambda$ given by Lemma \ref{lem:Mqs} and $p_\lambda$ is the rank of $\mathcal T_\lambda$. 
		
		First, we see that cluster characters associated to indecomposable postprojective modules can be computed recursively~:
		\begin{algorithm}[H]
			\caption{Recursive computation for postprojective modules}\label{algo:postprojective}
			\begin{algorithmic}[1]
				\REQUIRE $n \geq 0$, $i \in Q_0$
				\FOR {$j \in Q_0$}
				\STATE $X_{P_j[1]}:=u_j$
				\ENDFOR
				\STATE $$X_{P_i[-n]}:=\frac{\displaystyle \prod_{(j,m) \fl (i,-n)} X_{P_j[m]}}{X_{P_i[-n+1]}}$$
				\RETURN $X_{P_i[-n]}$
			\end{algorithmic}
		\end{algorithm}
		The algorithm is dual for indecomposable preinjective modules.

		Now that we are able to compute algorithmically characters associated to postprojective and preinjective modules, we give an algorithm to compute characters associated to indecomposable regular modules. 
		\begin{algorithm}[H]
			\caption{Computation of $X_{M_\lambda[k]^{(l)}}$ for $l \geq 1$, $k \in [0,p_\lambda-1]$  and $\lambda \in \P^1(\k)$}\label{algo:except}
			\begin{algorithmic}[1]
				\REQUIRE $\lambda \in \P^1(\k)$, $M \simeq M_\lambda[k]^{(l)}$ with $l \geq 1$ and $k \in [0, p_\lambda-1]$
				\FOR {$k \in [0, p_\lambda-1]$}
				\STATE $$X_{M_\lambda[k]}:=\frac{X_{B_\lambda[k]}+X_{B'_\lambda[k]}}{X_{S_e[k]}}$$ \COMMENT {Characters arising in the right-hand side can be computed using Algorithm \ref{algo:postprojective} and its dual version for preinjective modules} 
				\ENDFOR
				\STATE $X_{M}:=P_l(X_{M_\lambda[(k \mod p_\lambda)]}, \ldots, X_{M_\lambda[((k+l-1) \mod p_\lambda)]})$
				\RETURN $X_{M}$
			\end{algorithmic}
		\end{algorithm}

\section{An example~: Regular cluster variables in type $\Eaffine_6$}\label{section:varE6}
		Let $Q$ be the euclidean quiver of type $\Eaffine_6$ considered in Section \ref{ssection:E6affine}. For any $\lambda \in \P^1(\k)$, we provide an explicit expansion formula of the cluster variable corresponding to regular modules at the mouths of the exceptional tubes.
	
		In the tube $\mathcal T_0$ of rank two, we have~:
		\begin{align*}
			X_{N_0} & = 
				\frac{1}{u_1 u_2 u_4 u_6} (u_2 u_3 u_4 u_5 u_6 u_7 + u_1^3 + u_1^2 u_3 + u_1^2 u_5 \\
				& + u_1 u_3 u_5 + u_1^2 u_7 + u_1 u_3 u_7 + u_1 u_5 u_7 + u_3 u_5 u_7)
		\end{align*}
		and
		\begin{align*}
			X_{N_0[1]} &= 
				\frac 1{u_1^2 u_2 u_3 u_4 u_5 u_6 u_7} \left(u_2^2 u_3 u_4^2 u_5 u_6^2 u_7 \right.\\ 
				& + u_1 u_2 u_3 u_4 u_5 u_6^2 + u_1 u_2 u_3 u_4^2 u_6 u_7 + u_1 u_2^2 u_4 u_5 u_6 u_7 \\
				& + u_1^3 u_2 u_4 u_6 + u_1 u_2 u_3 u_4 u_5 u_6 + u_1 u_2 u_3 u_4 u_6 u_7 + u_1 u_2 u_4 u_5 u_6 u_7\\
				& + 2 u_2 u_3 u_4 u_5 u_6 u_7 + u_1^3 u_2 u_4 + u_1^3 u_2 u_6 + u_1^3 u_4 u_6 + u_1^2 u_3 u_4 u_6 \\
				& + u_1^2 u_2 u_5 u_6 + u_1^2 u_2 u_4 u_7 + u_1^3 u_2 + u_1^3 u_4 + u_1^2 u_3 u_4 + u_1^2 u_2 u_5 \\
				& + u_1^3 u_6 + u_1^2 u_3 u_6 + u_1^2 u_5 u_6 + u_1 u_3 u_5 u_6 + u_1^2 u_2 u_7 + u_1^2 u_4 u_7 \\
				& + u_1 u_3 u_4 u_7 + u_1 u_2 u_5 u_7 + u_1^3 + u_1^2 u_3 + u_1^2 u_5 + u_1 u_3 u_5 \\
				& \left. + u_1^2 u_7 + u_1 u_3 u_7 + u_1 u_5 u_7 + u_3 u_5 u_7 \right)
		\end{align*}

		In the tube $\mathcal T_1$ of rank three, we obtain~:
		$$X_{N_1} = \frac{u_2 u_3 u_4 u_6 u_7 + u_1^2 u_2 + u_1 u_2 u_7 + u_1^2 + u_1 u_3 + u_1 u_7 + u_3 u_7}{u_1 u_2 u_3 u_6}$$
		$$X_{N_1[1]} = \frac{u_2 u_3 u_4 u_5 u_6 + u_1^2 u_4 + u_1 u_3 u_4 + u_1^2 + u_1 u_3 + u_1 u_5 + u_3 u_5}{u_1 u_2 u_4 u_5}$$
		$$X_{N_1[2]} = \frac{u_2 u_4 u_5 u_6 u_7 + u_1^2 u_6 + u_1 u_5 u_6 + u_1^2 + u_1 u_5 + u_1 u_7 + u_5 u_7}{u_1 u_4 u_6 u_7}$$
		
		In the tube $\mathcal T_{\infty}$ of rank three, we get~:
		$$X_{N_{\infty}} = \frac{u_2 u_3 u_4 u_6 u_7 + u_1^2 u_6 + u_1 u_3 u_6 + u_1^2 + u_1 u_3 + u_1 u_7 + u_3 u_7}{u_1 u_2 u_6 u_7}$$
		$$X_{N_\infty[1]} = \frac{u_2 u_4 u_5 u_6 u_7 + u_1^2 u_4 + u_1 u_4 u_7 + u_1^2 + u_1 u_5 + u_1 u_7 + u_5 u_7}{u_1 u_4 u_5 u_6}$$
		$$X_{N_\infty[2]} = \frac{u_2 u_3 u_4 u_5 u_6 + u_1^2 u_2 + u_1 u_2 u_5 + u_1^2 + u_1 u_3 + u_1 u_5 + u_3 u_5}{u_1 u_2 u_3 u_4}$$

		This provides an interesting result in view of Fomin-Zelevinsky's positivity conjecture~:
		\begin{prop}
			Let $Q$ be an euclidean quiver of type $\Eaffine_6$ with the orientation considered in Section \ref{section:explicit}. Then the regular cluster variables of $\mathcal A(Q)$ belong to $\Z_{\geq 0}[\textbf u^{\pm 1}]$.
		\end{prop}
		\begin{proof}
			A regular cluster variable in $\mathcal A(Q)$ is by definition of the form $X_M$ for some indecomposable rigid regular $\kQ$-module $M$. Such an indecomposable regular rigid is of the form $R^{(l)}$ where $R$ is quasi-simple in an exceptional tube $\mathcal T_\lambda$ and $0<l<p_\lambda$. If $M$ is quasi-simple, the result holds according to the above expressions. Thus, the only case to consider is when $\lambda \in \ens{1, \infty}$ and $l=2$. In this case, the proof is inspired by \cite{Dupont:positiveregular}. According to the previous expressions, for any quasi-simple module $R$ in $\mathcal T_0$ or $\mathcal T_\infty$, the character $X_R$ can be written as 
			$$X_R = \sum_{i=1}^{n_R} L_{R,i}$$
			where $L_{R,i}$ is a (monic) Laurent monomial in the cluster $\textbf u$. Now, we observe that for any $R$, there exist integers $1 \leq i_0^R \leq n_R$ and $1 \leq j^R_0 \leq n_{R[-1]}$ such that 
			$$X_{R[-1]}=\frac{1}{L_{R,i_0}}+\sum_{i \neq j_0^R; \, i=1}^{n_{R[-1]}} L_{R[-1],i}.$$
			These monomials are explicitly listed in the following table.
			\begin{table}[H]
				$$\begin{array}{|c|c|c|c|c|c|}
					\hline
					R & L_{R,i^R_0} \\
					\hline
					N_1 &  \frac{u_4u_7}{u_1} \\
					\hline
					N_1[1] &  \frac{u_2u_5}{u_1} \\
					\hline
					N_1[2] &  \frac{u_3u_6}{u_1} \\
					\hline
					N_\infty &  \frac{u_3u_4}{u_1} \\
					\hline
					N_\infty[1] &  \frac{u_2u_7}{u_1} \\
					\hline
					N_\infty[2] &   \frac{u_5u_6}{u_1} \\
					\hline
				\end{array}$$
			\end{table}

			Thus, we get
			\begin{align*}
				X_{R^{(2)}} 
					& = P_2(X_R,X_{R[-1]}) \\
					& = X_RX_{R[-1]} - 1 \\
					& = \left(L_{R,i_0^R}+ \sum_{\substack{i=1 \\ i \neq i^R_0}}^{n_R} L_{R,i}\right)\left( \frac{1}{L_{R,i^R_0}}+\sum_{\substack{i=1\\i \neq j^R_0}}^{n_{R[-1]}} L_{R[-1],i} \right) - 1 \in \Z_{\geq 0}[\textbf u^{\pm 1}].
			\end{align*}

		\end{proof}

		We can similarly compute the character $X_{M_\lambda}$ associated to a quasi-simple module at the mouth of an homogeneous tube but the result is far too long to be presented here. It is a sum of 322 Laurent monomials so that it is in particular an element of $\Z_{\geq 0}[\textbf u^{\pm 1}]$. Written as an irreducible Laurent polynomial, its denominator is $\prod_i u_i^{\delta_i}$, accordingly to \cite[Theorem 3]{CK2}.

\section{Euler characteristics of complete grassmannians in type $\Eaffine$}\label{section:GrMtypeE}
	In type $\Eaffine_n$ with $n=6,7,8$, our algorithm allows to compute cluster characters associated to any object in the cluster category. The explicit Laurent expansions of these characters are so long that it would neither be reasonable nor useful to present them here. Nevertheless, our algorithms may be used to compute Euler characteristics of complete quiver grassmannians. In order to so so, it is enough to replace the lines ``$X_{P_j[1]}:=u_j$'' by  $X_{P_j[1]}:=1$ in Algorithm \ref{algo:postprojective} and its dual.

	In this section, we give the complete list of these Euler characteristics for quasi-simple modules in tubes of $\Gamma(\CC_Q)$ when $Q$ is an euclidean quiver of type $\Eaffine_n$ with $n=6,7,8$ equipped with one of the orientations considered in Section \ref{section:explicit}.

	In the following tables, we use the notations of Section \ref{section:explicit}. The omitted values correspond to $\tau$-periodic modules for which the value of the Euler characteristic of the complete quiver grassmannian can already be found in the table.
	\begin{table}[H]
		$$\begin{array}{|r|c|c|c|c|}
			\hline
			\lambda 	& 0  & 1 & \infty & \lambda \in \P^{\mathcal H} \\ \hline
			p_\lambda 	& 2  & 3 & 3 & 1 \\ \hline
			N_\lambda 	& 9  & 7 & 7 & 322 \\ \hline
			\tau N_\lambda 	& 36 & 7 & 7 & \\ \hline
			\tau^2 N_\lambda 	&    & 7 & 7 & \\ \hline
		\end{array}$$
		\caption{Characteristics of complete quiver grassmannians in type $\Eaffine_6$}
	\end{table}

	\begin{table}[H]
		$$\begin{array}{|r|c|c|c|c|}
			\hline
			\lambda 	& 0  & 1  & \infty & \lambda \in \P^{\mathcal H} \\ \hline
			p_\lambda 	& 3  & 4  & 2      & 1 \\ \hline
			N_\lambda 	& 10 & 7  & 61     & 3719 \\ \hline
			\tau N_\lambda 	& 9  & 9  & 61     & \\ \hline
			\tau^2 N_\lambda 	& 42 & 7  &        & \\ \hline
			\tau^3 N_\lambda 	&    & 9  &        & \\ \hline
		\end{array}$$
		\caption{Characteristics of complete quiver grassmannians in type $\Eaffine_7$}
	\end{table}

	\begin{table}[H]
		$$\begin{array}{|r|c|c|c|c|}
			\hline
			\lambda & 0 & 1 & \infty & \lambda \in \P^{\mathcal H} \\ \hline
			p_\lambda & 3 & 5 & 2 & 1 \\ \hline
			N_\lambda & 88 & 47 & 779 & 403520 \\ \hline
			\tau N_\lambda & 74 & 11 & 518 & \\ \hline
			\tau^2 N_\lambda & 62 & 9 & & \\ \hline
			\tau^3 N_\lambda & & 10 & & \\ \hline
			\tau^4 N_\lambda & & 9 & & \\ \hline
		\end{array}$$
		\caption{Characteristics of complete quiver grassmannians in type $\Eaffine_8$}
	\end{table}

	Note that the Euler characteristic of the complete quiver grassmannian of a $\kQ$-module $M$ is also equal to the number of Laurent monomials occurring in the Laurent expansion of $X_M$ in the initial cluster $\textbf u$ so that these numbers give an idea of the complexity of the associated variables.

\section*{Acknowledgements}
	This paper was written while the second author was at the university of Sherbrooke as a CRM-ISM postdoctoral fellow under the supervision of the first author, Thomas Br\"ustle and Virginie Charette.


\newcommand{\etalchar}[1]{$^{#1}$}

\end{document}